\documentclass[oneside,english,reqno]{amsart}

\usepackage[T1]{fontenc}
\usepackage[latin9]{inputenc}
\usepackage{babel}
\usepackage{amsthm}
\usepackage{amssymb}
\usepackage{graphicx}
\usepackage[unicode=true,pdfusetitle,
 bookmarks=true,bookmarksnumbered=false,bookmarksopen=false,
 breaklinks=false,pdfborder={0 0 1},backref=false,colorlinks=false]
 {hyperref}
\usepackage{breakurl}

\makeatletter
\numberwithin{equation}{section}
\numberwithin{figure}{section}
\theoremstyle{plain}
\newtheorem{thm}{\protect\theoremname}[section]
  \theoremstyle{plain}
  \newtheorem{algorithm}[thm]{\protect\algorithmname}
  \theoremstyle{remark}
  \newtheorem{rem}[thm]{\protect\remarkname}
  \theoremstyle{plain}
  \newtheorem{prop}[thm]{\protect\propositionname}
  \theoremstyle{definition}
  \newtheorem{defn}[thm]{\protect\definitionname}
  \theoremstyle{plain}
  \newtheorem{cor}[thm]{\protect\corollaryname}
  \theoremstyle{plain}
  \newtheorem{lem}[thm]{\protect\lemmaname}
  \theoremstyle{definition}
  \newtheorem{example}[thm]{\protect\examplename}
  \theoremstyle{remark}
  \newtheorem*{acknowledgement*}{\protect\acknowledgementname}

\newcommand{\bspan}{\mbox{\rm span}}

\makeatother

  \providecommand{\acknowledgementname}{Acknowledgement}
  \providecommand{\algorithmname}{Algorithm}
  \providecommand{\corollaryname}{Corollary}
  \providecommand{\definitionname}{Definition}
  \providecommand{\examplename}{Example}
  \providecommand{\lemmaname}{Lemma}
  \providecommand{\propositionname}{Proposition}
  \providecommand{\remarkname}{Remark}
\providecommand{\theoremname}{Theorem}

\begin{document}
\title[Projection onto intersection of affine spaces]{Accelerating the alternating projection algorithm for the case of affine subspaces using supporting hyperplanes}

\begin{abstract}
The von Neumann-Halperin method of alternating projections converges
strongly to the projection of a given point onto the intersection
of finitely many closed affine subspaces. We propose acceleration
schemes making use of two ideas: Firstly, each projection onto an
affine subspace identifies a hyperplane of codimension 1 containing
the intersection, and secondly, it is easy to project onto a finite
intersection of such hyperplanes. We give conditions for which our
accelerations converge strongly. Finally, we perform numerical experiments
to show that these accelerations perform well for a matrix model updating
problem. 
\end{abstract}

\author{C.H. Jeffrey Pang}

\curraddr{Department of Mathematics\\ 
National University of Singapore\\ 
Block S17 08-11\\ 
10 Lower Kent Ridge Road\\ 
Singapore 119076 }

\email{matpchj@nus.edu.sg}

\date{\today{}}


\maketitle
\tableofcontents{}

\section{Introduction}

Let $X$ be a (real) Hilbert space, and let $M_{1}$, $M_{2}$, $\dots$,
$M_{k}$ be a finite number of closed linear subspaces with $M:=\cap_{l=1}^{k}M_{l}$.
For any closed subspace $N$ of $X$, let $P_{N}$ denote the orthogonal
projection onto $N$. The von Neumann-Halperin method of alternating
projections, or MAP for short, is an iterative algorithm for determining
the best approximation $P_{M}x$, the projection of $x$ onto $M$.
We recall their theorem on the strong convergence of the MAP below.
\begin{thm}
\label{thm:1st-MAP}(von Neumann \cite{Neumann50} for $k=2$, Halperin
\cite{Halperin62} for $k\geq2$) Let $M_{1}$, $M_{2}$, $\dots$,
$M_{k}$ be closed subspaces in the Hilbert space $X$ and let $M:=\cap_{l=1}^{k}M_{l}$.
Then 
\begin{equation}
\lim_{n\to\infty}\|(P_{M_{k}}P_{M_{k-1}}\cdots P_{M_{1}})^{n}x-P_{M}x\|=0\mbox{ for all }x\in X.\label{eq:Halperin}
\end{equation}

\end{thm}
In the case where $k=2$, this result was rediscovered numerous times.

The method of alternating projections, as suggested in the formula
\eqref{eq:Halperin}, guarantees convergence to the projection $P_{M}x$,
but the convergence is slow in practice. Various acceleration schemes
have been studied in \cite{GPR67,GK89,BDHP03}. An identity for the
convergence of the method of alternating projections in the case of
linear subspaces is presented in \cite{Xu_Zikatanov02}.

We remark that the Boyle-Dykstra Theorem \cite{BD86} generalizes
the strong convergence to the projection in Theorem \ref{thm:1st-MAP}
to Dykstra's algorithm \cite{Dykstra83}, where the $M_{l}$ do not
have to be linear subspaces.

When the sets $M_{l}$ are not linear subspaces, a simple example
using a halfspace and a line in $\mathbb{R}^{2}$ shows that the method
of alternating projections may not converge to the projection $P_{M}x$.
Nevertheless, the method of alternating projections is still useful
for the SIP (Set Intersection Problem). When $M_{1}$, $M_{2}$, $\dots$,
$M_{k}$ is a finite number of closed (not necessarily convex) subsets
of a Hilbert space $X$, the SIP is the problem of finding a point
in $M:=\cap_{l=1}^{k}M_{l}$, i.e.,

\begin{equation}
\mbox{(SIP):}\quad\mbox{Find }x\in M:=\bigcap_{l=1}^{k}M_{l}\mbox{, where }M\neq\emptyset.\label{eq:SIP}
\end{equation}
An acceleration of the method of alternating projections for the case
where each $M_{l}$ were closed convex sets (but not necessarily subspaces)
was studied in \cite{cut_Pang12} and improved in \cite{better_alg}.
The idea there, named as the SHQP strategy (Supporting Halfspaces
and Quadratic Programming) was to store each of the halfspace produced
by the projection process, and use quadratic programming to project
onto an intersection of a reasonable number of these halfspaces. 

In the particular case of affine spaces, the SHQP strategy is even
easier to state and implement: Consider an affine space $M_{1}$ of
a Hilbert space $X$. First, the projection of a point $x_{0}$ onto
$M_{1}$ identifies the hyperplane of codimension 1 
\begin{equation}
\{x:\left\langle x_{0}-P_{M_{1}}x_{0},x\right\rangle =\left\langle x_{0}-P_{M_{1}}x_{0},P_{M_{1}}x_{0}\right\rangle \}\label{eq:1st-hyperplane}
\end{equation}
as a superset of $M_{1}$. Next, it is easy to project any point onto
the intersection of finitely many hyperplanes of the form \eqref{eq:1st-hyperplane}. 

A problem with many similarities but separate considerations and techniques
is that of \cite{NeedellTropp14}. In that paper, a randomized block
Kaczmarz method is analyzed.

\subsection{Contributions of this paper}

The techniques of \cite{cut_Pang12} gives additional assumptions
so that the SHQP strategy converges weakly to $P_{M}(x)$. The question
we ask in this paper is whether the SHQP strategy converges strongly
to the projection $P_{M}x$ in the case when $M:=\cap_{i=1}^{k}M_{i}$
and $M_{1}$, $M_{2}$, $\dots$, $M_{k}$ is a finite number of closed
affine subspaces like in the von Neumann-Halperin Theorem. We propose
Algorithm \ref{alg:accel}, which is based on a naive implementation
of the SHQP strategy. Based on the additional structure of affine
spaces, we propose Algorithm \ref{alg:accel2}, which is effective
when one of the affine subspaces is easy to project onto.

We prove that Algorithms \ref{alg:accel} and \ref{alg:accel2} converge
strongly to $P_{M}(x)$ under some assumptions in Section \ref{sec:Strong-conv}.
We also give reasons (Example \ref{exa:p-needed}) to explain why
these additional conditions cannot be removed. Our proof is adapted
from the proof of the Boyle-Dykstra Theorem \cite{BD86} on the strong
convergence of Dykstra's Algorithm \cite{Dykstra83} in the manner
presented in \cite{EsRa11}.

Next, we examine an implementation of our acceleration on a Matrix
Model Updating Problem (MMUP) from \cite[Section 6.2]{EsRa11}, who
in turn cited \cite{Datta_Sarkissian01,Moreno_Datta_Raydan09}. The
numerical experiments show the effectiveness of our algorithms.

\subsection{Notation}

We shall assume that $X$ is a Hilbert space with the inner product
$\left\langle \cdot,\cdot\right\rangle $ and norm $\|\cdot\|$.

\section{Algorithms}

In this section, we propose Algorithms \ref{alg:accel} and \ref{alg:accel2}
that seek to find the projection of a point onto the intersection
of a finite number of closed linear subspaces. It is clear to see
that our algorithms apply for affine spaces with nonempty intersection
as well (a fact we use in our experiments in Section \ref{sec:experiments}),
since a translation can reduce the problem to involving only linear
subspaces. 

We begin with our first algorithm.
\begin{algorithm}
\label{alg:accel}(Accelerated Projections) Let $M_{1}$, $M_{2}$,
$\dots$, $M_{k}$ be a finite number of closed linear subspaces in
a Hilbert space $X$. For a starting point $x_{0}\in X$, this algorithm
seeks to find $P_{M}(x_{0})$, where $M:=\cap_{l=1}^{k}M_{l}$.

\textbf{Step 0:} Set $i=0$.

\textbf{Step 1:} Project $x_{i}$ onto $M_{l_{i}}$, where $l_{i}\in\{1,\dots,k\}$,
to get $\tilde{x}_{i}$. This projection identifies a hyperplane $H_{i}:=\{x:\langle a_{i},x\rangle=b_{i}\}$,
where $a_{i}=x_{i}-P_{M_{l_{i}}}(x_{i})\in X$ and $b_{i}=\langle a_{i},P_{M_{l_{i}}}(x_{i})\rangle\in\mathbb{R}$,
such that $M\subset M_{l_{i}}\subset H_{i}$. (When $a_{i}=0$, then
$H_{i}=X$.)

\textbf{Step 2:} Choose $J_{i}\subset\{1,\dots,i\}$ such that $i\in J_{i}$,
and project $\tilde{x}_{i}$ onto $\tilde{H}_{i}:=\cap_{j\in J_{i}}H_{j}$
to get $x_{i+1}$. In short:
\begin{equation}
x_{i}\xrightarrow{P_{M_{l_{i}}}(\cdot)}\tilde{x}_{i}\xrightarrow{P_{\tilde{H}_{i}}(\cdot)}x_{i+1},\mbox{ with }M\subset\tilde{H}_{i}\mbox{ and }M\subset M_{l_{i}}\mbox{ for all }i.\label{eq:nested-projs-1}
\end{equation}

\textbf{Step 3: }The algorithm ends if some convergence criterion
is met. Otherwise, set $i\leftarrow i+1$ and return to step 1.

\end{algorithm}
\begin{rem}
(Limit points of $\{x_{i}\}$) We can easily figure out that 
\[
x_{i}-\tilde{x}_{i}\in M_{l_{i}}^{\perp}\subset M^{\perp}\mbox{ and }\tilde{x}_{i}-x_{i+1}\in\tilde{H}_{i}^{\perp}\subset M^{\perp},
\]
from which we can deduce that $x_{0}-x_{i}\in M^{\perp}$. Suppose
$\{x_{i}\}_{i}$ converges (weakly or strongly) to $\bar{x}$. We
can then deduce that $x_{0}-\bar{x}\in\overline{[\sum_{l=1}^{k}M_{l}^{\perp}]}=M^{\perp}$.
Furthermore, if $\bar{x}\in M$, the KKT conditions imply that $\bar{x}=P_{M}(x_{0})$.
For the rest of this paper, we will concentrate our efforts in showing
that in our algorithms, the iterates $\{x_{i}\}$ converge strongly
to $P_{M}(x_{0})$. 
\end{rem}
The next two easy results are preparation for Algorithm \ref{alg:accel2},
which is an improvement of Algorithm \ref{alg:accel} when one of
the linear subspaces, say $M_{1}$, is easy to project onto. Something
similar was done in \cite{Pierra84,BausCombKruk06}, where analytic
formulas for the projection onto an affine space and a halfspace were
derived.
\begin{prop}
\label{prop:proj-intersection-of-2}(Projection onto intersection
of affine spaces) Suppose $M$ and $\tilde{H}$ are linear subspaces
of a Hilbert space $X$ such that $\tilde{H}^{\perp}\subset M$. Then
$P_{M\cap\tilde{H}}(\cdot)=P_{\tilde{H}}\circ P_{M}(\cdot)$.\end{prop}
\begin{proof}
For $x\in X$, let $y:=P_{M}(x)$. Then $y-P_{\tilde{H}}(y)\in\tilde{H}^{\perp}\subset M$.
Since $y\in M$, we have $P_{\tilde{H}}\circ P_{M}(x)=P_{\tilde{H}}(y)\in M$.
It is also clear that $P_{\tilde{H}}\circ P_{M}(x)\in\tilde{H}$,
so $P_{\tilde{H}}\circ P_{M}(x)\in M\cap\tilde{H}$. 

Next, since $y-x\in M^{\perp}\subset[\tilde{H}\cap M]^{\perp}$, we
have 
\[
P_{\tilde{H}\cap M}(x)=P_{\tilde{H}\cap M}(y)=P_{\tilde{H}}(y)=P_{\tilde{H}}\circ P_{M}(x).
\]
Since the above holds for all $x\in X$, we are done. 
\end{proof}
\begin{figure}[h]
\includegraphics[scale=0.5]{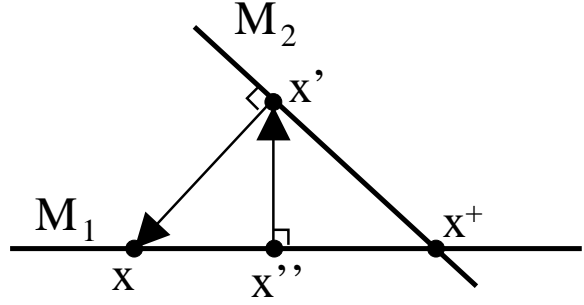}

\caption{\label{fig:aff} This figure illustrates the two dimensional space
spanned by $x$, $x^{\prime}$ and $x^{\prime\prime}$ in Proposition
\ref{prop:2-subspaces}. }
\end{figure}

\begin{prop}
\label{prop:2-subspaces}(2 subspaces) Let $M_{1}$ and $M_{2}$ be
two linear subspaces of a Hilbert space $X$. Suppose $x\in M_{1}$
and $x^{\prime}=P_{M_{2}}(x)$ and $x^{\prime\prime}=P_{M_{1}}(x^{\prime})$.
Then the hyperplane $H:=\left\{ \tilde{x}:\langle x-x^{\prime\prime},\tilde{x}\rangle=\left\langle x-x^{\prime\prime},x^{+}\right\rangle \right\} ,$
where $x^{+}:=x+\frac{\|x-x^{\prime}\|^{2}}{\|x-x^{\prime\prime}\|^{2}}(x^{\prime\prime}-x)$
(See Figure \ref{fig:aff}), is such that $M_{1}\cap M_{2}\subset H$.
Moreover, $x-x^{\prime\prime}$, the only vector in $H^{\perp}$ up
to a scalar multiple, satisfies $x-x^{\prime\prime}\in M_{1}$.\end{prop}
\begin{proof}
By the properties of projection, we have $M_{2}\subset H^{\prime}:=\{\tilde{x}:\langle x-x^{\prime},\tilde{x}-x^{\prime}\rangle=0\}$
and $M_{1}\subset H^{\prime\prime}:=\{\tilde{x}:\langle x^{\prime}-x^{\prime\prime},\tilde{x}-x^{\prime\prime}\rangle=0\}$.
By elementary geometry (See Figure \ref{fig:aff}), we can figure
out that the projection of $x^{\prime\prime}$ onto $H^{\prime}\cap H^{\prime\prime}$
is $x^{+}$. Thus $M_{1}\cap M_{2}\subset H^{\prime}\cap H^{\prime\prime}\subset H$,
from which the first part follows. The last sentence of the result
is clear.
\end{proof}
From the preparations in Propositions \ref{prop:proj-intersection-of-2}
and \ref{prop:2-subspaces}, we propose the following algorithm.
\begin{algorithm}
\label{alg:accel2}(Accelerated Projections 2) Let $M_{1}$, $M_{2}$,
$\dots$, $M_{k}$ be a finite number of closed linear subspaces in
a Hilbert space $X$. Suppose $M_{1}$ is easy to project onto. For
a starting point $x_{0}\in M_{1}$, this algorithm seeks to find $P_{M}(x_{0})$,
where $M:=\cap_{l=1}^{k}M_{l}$.

\textbf{Step 0:} Set $i=0$.

\textbf{Step 1:} Project $x_{i}$ onto $M_{l_{i}}$, where $l_{i}\in\{2,\dots,k\}$,
to get $x_{i}^{\prime}$. Project $x_{i}^{\prime}$ onto $M_{1}$
to get $x_{i}^{\prime\prime}$. This projection identifies a hyperplane
$H_{i}:=\{x:\langle a_{i},x\rangle=b_{i}\}$, where $a_{i}=x_{i}-x_{i}^{\prime\prime}\in M_{1}$
and $b_{i}=\langle a_{i},x_{i}+\frac{\|x_{i}-x_{i}^{\prime}\|^{2}}{\|x_{i}-x_{i}^{\prime\prime}\|^{2}}(x_{i}^{\prime\prime}-x_{i})\rangle\in\mathbb{R}$,
such that $M_{1}\cap M_{l_{i}}\subset H_{i}$.

\textbf{Step 2:} Choose $J_{i}\subset\{1,\dots,i\}$ such that $i\in J_{i}$,
and set $x_{i+1}=P_{\tilde{H}_{i}}(x_{i})$, which also equals $P_{\tilde{H}_{i}}(x_{i})$
since $x_{i}-x_{i}^{\prime\prime}\in H_{i}^{\perp}\subset\tilde{H}_{i}^{\perp}$.
One has 
\begin{equation}
x_{i}\xrightarrow{P_{M_{l_{i}}}(\cdot)}x_{i}^{\prime}\xrightarrow{P_{M_{1}}(\cdot)}x_{i}^{\prime\prime}\xrightarrow{P_{\tilde{H}_{i}}(\cdot)=P_{M_{1}\cap\tilde{H}_{i}}(\cdot)}x_{i+1},\label{eq:nested-projs}
\end{equation}
with $M\subset\tilde{H}_{i}$ and $M\subset M_{l_{i}}$ for all $i$.

\textbf{Step 3: }The algorithm ends if some convergence criterion
is met. Otherwise, set $i\leftarrow i+1$ and return to step 1.

\end{algorithm}
If $x_{0}\notin M_{1}$ in Algorithm \ref{alg:accel2}, we can start
the algorithm with $x_{0}\leftarrow P_{M_{1}}(x_{0})$ instead. It
is clear that $P_{M}(x_{0})=P_{M}(P_{M_{1}}(x_{0}))$. So if the algorithm
with the adjusted starting point converges to $P_{M}(P_{M_{1}}(x_{0}))$,
then it converges to $P_{M}(x_{0})$. 

We show that $x_{i}\in M_{1}$ for all $i$ and also explain why $P_{\tilde{H}_{i}}(x_{i}^{\prime\prime})=P_{M_{1}\cap\tilde{H}_{i}}(x_{i}^{\prime\prime})$
in \eqref{eq:nested-projs}. The assumptions state that $x_{0}\in M_{1}$.
Suppose $x_{i}\in M_{1}$. Then 
\[
x_{i+1}=P_{\tilde{H}_{i}}(x_{i})=P_{\tilde{H}_{i}}\circ P_{M_{1}}(x_{i})=P_{M_{1}\cap\tilde{H}_{i}}(x_{i}).
\]
The last equation comes from applying the fact that $\tilde{H}_{i}^{\perp}=\sum_{j\in J_{i}}H_{j}^{\perp}$
and $H_{j}^{\perp}\subset M_{1}$ for all $j$ from Proposition \ref{prop:2-subspaces}
onto Proposition \ref{prop:proj-intersection-of-2}. The formula \eqref{eq:nested-projs}
is a useful tool in the analysis of Algorithm \ref{alg:accel2}.

The linear subspace $M_{1}\cap\tilde{H}_{i}$ in Algorithm \ref{alg:accel2}
has a larger codimension than the $\tilde{H_{i}}$ in Algorithm \ref{alg:accel}.
Thus operations involving the projection $P_{M_{1}\cap\tilde{H}_{i}}(\cdot)$
can be expected to get iterates closer to $M$ than $P_{\tilde{H}_{i}}(\cdot)$.
So when $M_{1}$ is easy to project onto, we expect Algorithm \ref{alg:accel2}
to converge in fewer iterations and less time than Algorithm \ref{alg:accel}.
Such a condition is met for the example we present in Section \ref{sec:experiments},
and we will see that Algorithm \ref{alg:accel2} is indeed better.
Another factor that may play a role in the fast convergence observed
is that $M_{1}$ has a larger codimension than the other subspaces.

\section{\label{sec:Strong-conv}Strong convergence results}

In this section, we prove the strong convergence results for Algorithms
\ref{alg:accel} and \ref{alg:accel2}. 

We recall some easy results on the projection onto a closed linear
subspace and Fej\'{e}r monotonicity.
\begin{thm}
\label{thm:1-attractive}(Orthogonal projection onto linear subspaces)
Let $X$ be a Hilbert space, and suppose $T:X\to X$ is a projection
of a point $x$ onto a closed linear subspace $S$. Then 
\[
\|x-Tx\|^{2}=\|x\|^{2}-\|Tx\|^{2}\mbox{ for all }x\in X.
\]
\end{thm}
\begin{defn}
(Fej\'{e}r monotone sequence) Let $X$ be a Hilbert space, $C\subset X$
be a closed convex set, and $\{x_{i}\}$ be a sequence in $X$. We
say that $\{x_{i}\}$ is\emph{ Fej\'{e}r monotone with respect to
$C$} if 
\[
\|x_{i+1}-c\|\leq\|x_{i}-c\|\mbox{ for all }c\in C\mbox{ and }i=1,2,\dots
\]

\end{defn}
A tool for obtaining a Fej\'{e}r monotone sequence is stated below.
\begin{thm}
\label{thm:Fejer-contraction}(Fej\'{e}r attraction property) Let
$X$ be a Hilbert space. For a closed convex set $C\subset X$, $x\in X$,
$\lambda\in[0,2]$, and the projection $P_{C}(x)$ of $x$ onto $C$,
let the \emph{relaxation operator }$R_{C,\lambda}:X\to X$ \cite{Agmon54}
be defined by 
\[
R_{C,\lambda}(x)=x+\lambda(P_{C}(x)-x).
\]
Then 
\[
\|R_{C,\lambda}(x)-c\|^{2}\leq\|x-c\|^{2}-\lambda(2-\lambda)d(x,C)^{2}\mbox{ for all }y\in C.
\]
(For this paper, we only consider the case $\lambda=1$, which corresponds
to the projection.)
\end{thm}
We need a few lemmas proven in \cite{BD86} and a few classical results
used in \cite{BD86} for the proof of our result.
\begin{thm}
\label{thm:Unif-bddness}(Uniform boundedness principle) Let $\{f_{n}\}$
be a sequence of continuous linear functionals on a Hilbert space
$X$ such that $\sup_{n}|f_{n}(x)|<\infty$ for each $x\in X$. Then
$\|f_{n}\|\leq K<\infty$.\end{thm}
\begin{cor}
\label{cor:Unif-bddness-2}Let $\{f_{n}\}$ be a sequence of linear
functionals on a Hilbert space $X$ such that for each $x\in X$,
$\{f_{n}(x)\}$ converges. Then there is a continuous linear functional
$f$ such that $f(x)=\lim_{n}f_{n}(x)$ and $\|f\|\leq\liminf_{n}\|f_{n}\|$.\end{cor}
\begin{thm}
\label{thm:Kadec-Klee}(Kadec-Klee property) In a Hilbert space, $x_{n}\to x$
strongly if and only if $x_{n}\to x$ weakly and $\|x_{n}\|\to\|x\|$.
\begin{thm}
\label{thm:Banach-Saks}(Banach-Saks property) Let $\{x_{n}\}$ be
a sequence in a Hilbert space that converges weakly to $x$. Then
we can find a subsequence $\{x_{n_{k}}\}$ such that the arithmetic
mean $\frac{1}{m}\sum_{k=1}^{m}x_{n_{k}}$ converges strongly to $x$.
\end{thm}
\end{thm}
\begin{lem}
\cite{BD86}\label{lem:Sum-of-squares}(Sum of squares) Suppose a
sequence of nonnegative numbers $\{a_{j}\}_{j}$ is such that $\sum_{j=1}^{\infty}a_{j}^{2}$
converges. Then there is a subsequence $\{j_{t}\}_{t}$ such that
the sequence $\{\sum_{s=1}^{j_{t}}a_{s}a_{j_{t}}\}_{t}$ converges
to zero.
\end{lem}
We now prove a result that will be used in all the variants of our
strong convergence results for Algorithms \ref{alg:accel} and \ref{alg:accel2}.
This result is modified from that of \cite{BD86}, and we follow the
treatment in \cite{EsRa11}.
\begin{prop}
\label{prop:4-condns}(Conditions for strong convergence) Let $\{M_{l}\}_{l=1}^{k}$
be linear subspaces of a Hilbert space $X$, and $M:=\cap_{l=1}^{k}M_{l}$.
For a starting $x_{0}\in X$, suppose that the iterates $\{x_{i}\}_{i}$
generated by an algorithm satisfy
\begin{enumerate}
\item $\{x_{i}\}_{i}$ is Fej\'{e}r monotone with respect to $M$.
\item There exists a subsequence $\{i_{t}\}_{t}$ such that 
\begin{equation}
\limsup_{t\to\infty}\langle x_{i_{t}}-x_{0},x_{i_{t}}\rangle\leq0.\label{eq:prop-ppty2}
\end{equation}

\item For all $l\in\{1,\dots,k\}$ and $i>0$, there is some $p_{i}^{l}>0$
such that $x_{i+p_{i}^{l}}\in M_{l}$ and $\lim_{i\to\infty}\|x_{i}-x_{i+p_{i}^{l}}\|=0$.
\item $x_{i}-x_{0}\in M^{\perp}$ for all $i$.
\end{enumerate}
Then the sequence of iterates $\{x_{i}\}$ converges strongly to $P_{M}(x_{0})$.\end{prop}
\begin{proof}
The proof of this result is modified from that of \cite{BD86}, following
the presentation in \cite{EsRa11}. By property (2), we choose a subsequence
satisfying \eqref{eq:prop-ppty2}. By property (1), $\{x_{i_{t}}\}_{t}$
is a bounded sequence, so we can assume, by finding a subsequence
if necessary, that the weak limit 
\[
h:=\mbox{w-}\lim_{t\to\infty}x_{i_{t}}
\]
exists. Property (3) states that for each $l\in\{1,\dots,k\}$, we
can find a sequence $\{p_{i}^{l}\}_{i}\subset[0,\infty)$ such that
$x_{i+p_{i}^{l}}\in M_{l}$ and $\lim_{i\to\infty}\|x_{i}-x_{i+p_{i}^{l}}\|=0$.
We therefore have 
\[
\mbox{w-}\lim_{t\to\infty}x_{i_{t}+p_{i_{t}}^{l}}=\mbox{w-}\lim_{t\to\infty}x_{i_{t}}=h.
\]
The Banach-Saks Property (Theorem \ref{thm:Banach-Saks}) implies
that we can further choose a subsequence of $\{x_{i_{t}}\}_{t}$ if
necessary (we don't relabel) so that 
\begin{equation}
\frac{1}{m}\sum_{t=1}^{m}x_{i_{t}+p_{i_{t}}^{l}}\mbox{ converges strongly to }h\mbox{ as }m\to\infty.\label{eq:conv-comb-in-M-l-1}
\end{equation}
The term on the left of \eqref{eq:conv-comb-in-M-l-1} lies in $M_{l}$.
Since $l$ is arbitrary, we conclude that $h\in M$. Since $\{\|x_{i_{t}}\|\}_{t}$
is bounded, we can choose a subsequence if necessary so that 
\[
u:=\lim_{t\to\infty}\|x_{i_{t}}\|.
\]
By applying the Uniform Boundedness Principle (Theorem \ref{thm:Unif-bddness}
and Corollary \ref{cor:Unif-bddness-2}), we have 
\begin{equation}
\|h\|\leq\liminf_{t\to\infty}\|x_{i_{t}}\|=\lim_{t\to\infty}\|x_{i_{t}}\|=u.\label{eq:weak-conv-ineq-1}
\end{equation}
Since $x_{i_{t}}-x_{0}\in M^{\perp}$, we have \textrm{$\left\langle x_{i_{t}}-x_{0},y\right\rangle =0$
}for all\textrm{ $y\in M$. }So for all $y\in M$,
\begin{eqnarray}
0 & \geq & \limsup_{t\to\infty}\left\langle x_{i_{t}}-x_{0},x_{i_{t}}\right\rangle \nonumber \\
 & = & \limsup_{t\to\infty}\left\langle x_{i_{t}}-x_{0},x_{i_{t}}-y\right\rangle \nonumber \\
 & = & u^{2}-\left\langle h,y\right\rangle -\left\langle x_{0},h\right\rangle +\left\langle x_{0},y\right\rangle \label{eq:h-u-ineq-1}\\
 & \geq & \left\langle h-x_{0},h-y\right\rangle .\nonumber 
\end{eqnarray}
This means that $h=P_{M}x_{0}$. Next, we use \eqref{eq:h-u-ineq-1}
and substitute $y=h$ to get $0\geq u^{2}-\|h\|^{2}$, which together
with \eqref{eq:weak-conv-ineq-1}, gives $u=\|h\|$. By the Kadec-Klee
property (Theorem \ref{thm:Kadec-Klee}), we conclude that the subsequence
$\{x_{i_{t}}\}_{t}$ converges strongly to $h$. 

To see that $\{x_{i}\}_{i}$ converges strongly to $h$, we make use
of the Fej\'{e}r monotonicity of the iterates with respect to $M$
and $h\in M$.\end{proof}
\begin{rem}
\label{rem:condns-1-4}(Conditions (1) and (4) of Proposition \ref{prop:4-condns})
The sequence we apply Proposition \ref{prop:4-condns} on for our
next results on Algorithm \ref{alg:accel} is actually $x_{0},$ $\tilde{x}_{0},$
$x_{1},$ $\tilde{x}_{1},$ $x_{2},$ \textrm{$\tilde{x}_{2},$} $\dots$
instead of $\{x_{i}\}$. Similarly, the sequence we apply Proposition
\ref{prop:4-condns} on for our next results on Algorithm \ref{alg:accel}
is actually $x_{0},$ $x_{0}^{\prime},$ $x_{0}^{\prime\prime},$
$x_{1},$ $x_{1}^{\prime},$ $x_{1}^{\prime\prime},$ $x_{2},$ $x_{2}^{\prime},$
$x_{2}^{\prime\prime},$ $\dots$. We remark that for Algorithm \ref{alg:accel},
condition (1) holds because of \eqref{eq:nested-projs-1}. Similarly,
in Algorithm \ref{alg:accel2}, condition (1) holds due to \eqref{eq:nested-projs}.
Condition (4) holds for Algorithm \ref{alg:accel} because \eqref{eq:nested-projs-1}
implies that 
\begin{eqnarray*}
 &  & x_{i}-\tilde{x}_{i\phantom{+1}}\in M_{l_{i}}^{\perp}\subset M^{\perp}\\
 & \mbox{and } & \tilde{x}_{i}-x_{i+1}\in\tilde{H}_{i}^{\perp}\,\subset M^{\perp},
\end{eqnarray*}
from which we can easily deduce $x_{0}-x_{i}\in M^{\perp}$ and $x_{0}-\tilde{x}_{i}\in M^{\perp}$
for all $i$ as needed. The analysis for Algorithm \ref{alg:accel2}
is similar.
\end{rem}
We now prove the convergence of Algorithm \ref{alg:accel} for the
easier case first.
\begin{thm}
\label{thm:conv-alg-1-ver-1}(Strong convergence of Algorithm \ref{alg:accel}:
Version 1) Suppose that in Algorithm \ref{alg:accel}, the additional
conditions are satisfied:
\begin{enumerate}
\item [(A)]There is a number $\bar{p}$ such that for all $l\in\{1,\dots,k\}$
and $i>0$, there is a $p_{i}^{l}\in[0,\bar{p}]$ such that $\tilde{x}_{i+p_{i}^{l}}=P_{M_{l}}(x_{i+p_{i}^{l}})$. 
\item [(B)]The hyperplanes $\tilde{H}_{i}$ are chosen such that $x_{0}-x_{i}\in\bspan(\{a_{j}:j\in J_{i}\})$
for all iterations $i$. 
\end{enumerate}
Then the sequence of iterates $\{x_{i}\}_{i}$ converges strongly
to $P_{M}(x_{0})$. \end{thm}
\begin{proof}
We apply Proposition \ref{prop:4-condns}. The sequence we apply Proposition
\ref{prop:4-condns} to is actually $x_{0},\tilde{x}_{0},x_{1},\tilde{x}_{1},x_{2},\tilde{x}_{2},\dots$
instead of $\{x_{i}\}$. By Remark \ref{rem:condns-1-4}, it suffices
to check conditions (2) and (3) of Proposition \ref{prop:4-condns}. 

\textbf{Step 1: Condition (A) implies Condition (3) of Proposition
\ref{prop:4-condns}.}

By condition (A), for any $i>0$ and $l\in\{1,\dots,k\}$, there exists
a $p_{i}^{l}\in[0,\bar{p}]$ such that $\tilde{x}_{i+p_{i}^{l}}\in M_{l}$.
By using Theorem \ref{thm:1-attractive} repeatedly, we have 
\[
\sum_{i=0}^{\infty}[\|x_{i}-\tilde{x}_{i}\|^{2}+\|\tilde{x}_{i}-x_{i+1}\|^{2}]\leq\|x_{0}\|^{2}<\infty.
\]
Therefore the sequence $\|x_{0}-\tilde{x}_{0}\|,$ $\|\tilde{x}_{0}-x_{1}\|,$
$\|x_{1}-\tilde{x}_{1}\|,$ $\dots$ converges to zero. Since 
\[
\|x_{i}-\tilde{x}_{i+p_{i}^{l}}\|\leq\|x_{i+p_{i}^{l}}-\tilde{x}_{i+p_{i}^{l}}\|+\sum_{j=0}^{\bar{p}-1}[\|x_{i+j}-\tilde{x}_{i+j}\|+\|\tilde{x}_{i+j}-x_{i+j+1}\|],
\]
we see that $\{\|x_{i}-\tilde{x}_{i+p_{i}^{l}}\|\}_{i}$ is bounded
by a finite sum of terms with limit zero. Hence $\|x_{i}-\tilde{x}_{i+p_{i}^{l}}\|\to0$
as $i\to\infty$. Thus condition (3) holds.

\textbf{Step 2: Condition (B) implies Condition (2) of Proposition
\ref{prop:4-condns}. }

We prove 
\begin{equation}
\langle x_{0}-x_{i},x_{i}\rangle=0\mbox{ for all }i>0,\label{eq:x0-xi-xi-condn}
\end{equation}
which clearly implies Condition (2). We use standard induction. It
is easy to check that formula \eqref{eq:x0-xi-xi-condn} holds for
$i=1$. Suppose it holds for $i=i^{*}$. We want to show that it holds
for $i=i^{*}+1$. We have $x_{i^{*}+1}=P_{\tilde{H}_{i^{*}}}(\tilde{x}_{i^{*}})$,
or equivalently, $x_{i^{*}+1}\in\tilde{x}_{i^{*}}+\tilde{H}_{i^{*}}^{\perp}$.
Since $i^{*}\in J_{i^{*}}$, we have $x_{i^{*}}-\tilde{x}_{i^{*}}\in\tilde{H}_{i^{*}}^{\perp}$,
so $x_{i^{*}}\in x_{i^{*}+1}+\tilde{H}_{i^{*}}^{\perp}$, or $x_{i^{*}+1}\in P_{\tilde{H}_{i^{*}}}(x_{i^{*}})$.
Since $x_{0}-x_{i^{*}}\in\bspan(\{a_{j}:j\in J_{i^{*}}\})$, we have
$x_{0}-x_{i^{*}}\in\tilde{H}_{i^{*}}^{\perp}$, so $x_{i^{*}+1}\in P_{\tilde{H}_{i^{*}}}(x_{0})$
using a similar argument. Since $0\in\tilde{H}_{i^{*}}$, we can deduce
\eqref{eq:x0-xi-xi-condn}, ending our proof by induction.
\end{proof}
Note that condition (A) of Theorem \ref{thm:conv-alg-1-ver-1} satisfied
in the classical method of alternating projections, but condition
(B) is not. We propose a second convergence result such that includes
the classical method of alternating projections. For the iterates
in Algorithm \ref{alg:accel}, $i\geq0$ and $l\in\{1,\dots,k\}$,
we define $\tilde{v}_{i,l}\in X$ and $v_{i,l}\in X$ to be such that
\begin{eqnarray*}
 &  & x_{i}-\tilde{x}_{i\phantom{+1}}=\sum_{l=1}^{k}\tilde{v}_{i,l},\\
 & \mbox{ and } & \tilde{x}_{i}-x_{i+1}=\sum_{l=1}^{k}v_{i,l},\mbox{ where }\tilde{v}_{i,l},v_{i,l}\in M_{l}^{\perp}\mbox{ for all }l\in\{1,\dots,k\}.
\end{eqnarray*}
Such a representation is not unique. This part of the proof is modified
from the treatment in \cite{EsRa11} of \cite{BD86}.
\begin{thm}
(Strong convergence of Algorithm \ref{alg:accel}: Version 2) Suppose
that in Algorithm \ref{alg:accel}, the additional conditions are
satisfied:
\begin{enumerate}
\item [(A)]There is a number $\bar{p}$ such that for all $l\in\{1,\dots,k\}$
and $i>0$, there is a $p_{i}^{l}\in[0,\bar{p}]$ such that $\tilde{x}_{i+p_{i}^{l}}=P_{M_{l}}(x_{i+p_{i}^{l}})$. 
\item [(B${}^{\prime}$)] There is a number $K$ such that \textrm{\textup{
\begin{equation}
\sum_{l=1}^{k}[\|\tilde{v}_{j,l}\|^{2}+\|v_{j,l}\|^{2}]\leq K[\|x_{i}-\tilde{x}_{i}\|^{2}+\|\tilde{x}_{i}-x_{i+1}\|^{2}]\mbox{ for all }i\geq0.\label{eq:B-prime}
\end{equation}
}}
\end{enumerate}
Then the iterates $\{x_{i}\}$ converge strongly to $P_{M}(x_{0})$. \end{thm}
\begin{proof}
Like in Theorem \ref{thm:conv-alg-1-ver-1}, we apply Proposition
\ref{prop:4-condns}. The sequence we apply Proposition \ref{prop:4-condns}
on is actually $x_{0},\tilde{x}_{0},x_{1},\tilde{x}_{1},x_{2},\tilde{x}_{2},\dots$
instead of $\{x_{i}\}$. The proof that condition (A) implies condition
(3) of Proposition \ref{prop:4-condns} is the same as that in Theorem
\ref{thm:conv-alg-1-ver-1}. We proceed with the rest of the proof.

\textbf{Step 1: Condition (2) of Proposition \ref{prop:4-condns}
holds.}

By using Theorem \ref{thm:1-attractive} repeatedly, we have
\begin{equation}
\sum_{i=0}^{\infty}[\|x_{i}-\tilde{x}_{i}\|^{2}+\|\tilde{x}_{i}-x_{i+1}\|^{2}]\leq\|x_{0}\|^{2}<\infty.\label{eq:sum-sq-finite}
\end{equation}
For $j\in\mathbb{N}_{0}$, define $\alpha_{j}\in\mathbb{R}$ to be
\[
\alpha_{j}:=\sum_{l=1}^{k}[\|\tilde{v}_{j,l}\|+\|v_{j,l}\|].
\]
By repeatedly using the inequality $2cd\leq c^{2}+d^{2}$ onto the
expansion of $\alpha_{j}^{2}$ and \eqref{eq:B-prime}, we have 
\[
\alpha_{j}^{2}\leq2k\sum_{l=1}^{k}[\|\tilde{v}_{j,l}\|^{2}+\|v_{j,l}\|^{2}]\leq2kK[\|x_{i}-\tilde{x}_{i}\|^{2}+\|\tilde{x}_{i}-x_{i+1}\|^{2}].
\]
In view of \eqref{eq:sum-sq-finite}, the sum $\sum_{j=0}^{\infty}\alpha_{j}^{2}$
is finite.

Next, we calculate the bounds on the inner product $\langle x_{i}-x_{0},x_{i}\rangle$.
By Condition (A), for each $l\in\{1,\dots,k\}$ and $i>0$, there
is some $p_{i}^{l}\in[0,\bar{p}]$ such that $\tilde{x}_{i+p_{i}^{l}}=P_{M_{l}}(x_{i+p_{i}^{l}})$,
from which we get $\tilde{x}_{i+p_{i}^{l}}\in M_{l}$. 

Since $x_{i}-x_{0}=\sum_{s=0}^{i}\sum_{l=1}^{k}[\tilde{v}_{s,l}+v_{s,l}]$
and $[\tilde{v}_{s,l}+v_{s,l}]\in M_{l}^{\perp}$, we have 
\[
\langle x_{i}-x_{0},x_{i}\rangle=\sum_{s=0}^{i-1}\sum_{l=1}^{k}\langle\tilde{v}_{s,l}+v_{s,l},x_{i}\rangle=\sum_{s=0}^{i-1}\sum_{l=1}^{k}\langle\tilde{v}_{s,l}+v_{s,l},x_{i}-\tilde{x}_{i+p_{i}^{l}}\rangle.
\]
Since 
\begin{eqnarray*}
\|x_{i}-\tilde{x}_{i+p_{i}^{l}}\| & \leq & \sum_{s=0}^{\bar{p}}[\|x_{i+s}-\tilde{x}_{i+s}\|+\|\tilde{x}_{i+s}-x_{i+s+1}\|]\\
 & \leq & \sum_{s=0}^{\bar{p}}\sum_{l=1}^{k}[\|\tilde{v}_{i+s}\|+\|v_{i+s}\|]\\
 & = & \sum_{s=0}^{\bar{p}}\alpha_{i+s},
\end{eqnarray*}
we continue the earlier calculations to get 
\[
\langle x_{i}-x_{0},x_{i}\rangle\leq\sum_{s=0}^{i}\underbrace{\sum_{l=1}^{k}[\|\tilde{v}_{s,l}\|+\|v_{s,l}\|]}_{=\alpha_{s}}\|x_{i}-\tilde{x}_{i+p_{i}^{l}}\|\leq\left[\sum_{s=0}^{i}\alpha_{s}\right]\left[\sum_{s=0}^{\bar{p}}\alpha_{i+s}\right].
\]
Define $\beta_{j}:=\sum_{s=0}^{\bar{p}}\alpha_{j[\bar{p}+1]+s}$.
The inequality above would imply 
\[
\langle x_{j[\bar{p}+1]}-x_{0},x_{j[\bar{p}+1]}\rangle\leq\sum_{s=0}^{j}\beta_{s}\beta_{j}.
\]
Since $\beta_{j}^{2}\leq[\bar{p}+1]\sum_{s=0}^{\bar{p}}\alpha_{j[\bar{p}+1]+s}^{2}$,
we see that $\sum_{j=1}^{\infty}\beta_{j}^{2}\leq[\bar{p}+1]\sum_{j=1}^{\infty}\alpha_{j}^{2}<\infty$.
By Lemma \ref{lem:Sum-of-squares}, we can find a subsequence $\{i_{t}\}$
such that $\limsup_{t\to\infty}\langle x_{i_{t}}-x_{0},x_{i_{t}}\rangle\leq0$,
which is exactly condition (2). Thus we are done.
\end{proof}
In the case of alternating projections, it is clear to see that condition
(B$^{\prime}$) is satisfied with $K=1$ because $\tilde{x}_{i}-x_{i+1}=0$
and $v_{i,l}=0$ for all $i\geq0$, and for each $i\geq0$, only one
of the $\tilde{v}_{i,l}$ among $l\in\{1,\dots,k\}$ equals to $x_{i}-\tilde{x}_{i}$,
and the rest of the $\tilde{v}_{i,l}$ are zero. 

We remark that the condition (B$^{\prime}$) can be checked once
we get the new iterate $\tilde{x}_{j}^{(i)}$. The value $K$ can
be chosen to be any finite value.

We now proceed to prove a strong convergence result of Algorithm \ref{alg:accel2}.
The proof is similar to that of Theorem \ref{thm:conv-alg-1-ver-1},
but we shall include the details for completeness.
\begin{thm}
\label{thm:conv-alg-2-ver-1} (Strong convergence of Algorithm \ref{alg:accel2})
Suppose that in Algorithm \ref{alg:accel2}, the additional conditions
are satisfied:
\begin{enumerate}
\item [(A)]There is a number $\bar{p}$ such that for all $l\in\{2,\dots,k\}$
and $i>0$, there is a $p_{i}^{l}\in[0,\bar{p}]$ such that $x_{i+p_{i}^{l}}^{\prime}=P_{M_{l}}(x_{i+p_{i}^{l}})$. 
\item [(B)]The hyperplanes $\tilde{H}_{i}$ are chosen such that $x_{0}-x_{i}\in\bspan(\{a_{j}:j\in J_{i}\})$
for all iterations $i$. 
\end{enumerate}
Then the sequence of iterates $\{x_{i}\}_{i}$ converges strongly
to $P_{M}(x_{0})$. \end{thm}
\begin{proof}
We apply Proposition \ref{prop:4-condns}. The sequence we apply Proposition
\ref{prop:4-condns} to is actually $x_{0},$ $x_{0}^{\prime},$ $x_{0}^{\prime\prime},$
$x_{1},$ $x_{1}^{\prime},$ $x_{1}^{\prime\prime},$ $x_{2},$ $x_{2}^{\prime},$
$x_{2}^{\prime\prime},$ $\dots$. instead of $\{x_{i}\}$. By Remark
\ref{rem:condns-1-4}, it suffices to check conditions (2) and (3)
of Proposition \ref{prop:4-condns}. The changes from the proof of
Theorem \ref{thm:conv-alg-1-ver-1} are minimal, but we still include
details for completeness.

\textbf{Step 1: Condition (A) implies Condition (3) of Proposition
\ref{prop:4-condns}.}

By condition (A), for any $i>0$ and $l\in\{2,\dots,k\}$, there exists
a $p_{i}^{l}\in[0,\bar{p}]$ such that $x_{i+p_{i}^{l}}^{\prime}\in M_{l}$.
Note that $x_{i+1}=P_{M_{1}\cap\tilde{H}_{i}}(x_{i}^{\prime})$ by
Proposition \ref{prop:proj-intersection-of-2}. By using Theorem \ref{thm:1-attractive}
repeatedly, we have 
\[
\sum_{i=0}^{\infty}[\|x_{i}-x_{i}^{\prime}\|^{2}+\|x_{i}^{\prime}-x_{i+1}\|^{2}]\leq\|x_{0}\|^{2}<\infty.
\]
Therefore the sequence $\|x_{0}-x_{0}^{\prime}\|,$ $\|x_{0}^{\prime}-x_{1}\|,$
$\|x_{1}-x_{1}^{\prime}\|,$ $\dots$ converges to zero. Since 
\[
\|x_{i}-x_{i+p_{i}^{l}}^{\prime}\|\leq\|x_{i+p_{i}^{l}}-x_{i+p_{i}^{l}}^{\prime}\|+\sum_{j=0}^{\bar{p}}[\|x_{i+j}-x_{i+j}^{\prime}\|+\|x_{i+j}^{\prime}-x_{i+j+1}\|],
\]
it is clear that $\|x_{i}-x_{i+p_{i}^{l}}^{\prime}\|\to0$ as $i\to\infty$.
Thus condition (3) holds.

\textbf{Step 2: Condition (B) implies Condition (2) of Proposition
\ref{prop:4-condns}. }

We prove 
\begin{equation}
\langle x_{0}-x_{i},x_{i}\rangle=0\mbox{ for all }i>0,\label{eq:x0-xi-xi-condn-1}
\end{equation}
which clearly implies Condition (2). We use standard induction. It
is easy to check that formula \eqref{eq:x0-xi-xi-condn-1} holds for
$i=1$. Suppose it holds for $i=i^{*}$. We want to show that it holds
for $i=i^{*}+1$. We have $x_{i^{*}+1}=P_{M_{1}\cap\tilde{H}_{i^{*}}}(x_{i^{*}}^{\prime})$,
or equivalently, 
\[
x_{i^{*}+1}\in x_{i^{*}}^{\prime}+[M_{1}\cap\tilde{H}_{i^{*}}]^{\perp}.
\]
Since $x_{i^{*}}^{\prime\prime}=P_{M_{1}}(x_{i^{*}}^{\prime})$, we
have $x_{i^{*}}^{\prime}-x_{i^{*}}^{\prime\prime}\in M_{1}^{\perp}\subset[M_{1}\cap\tilde{H}_{i^{*}}]^{\perp}$.
Next, since $i^{*}\in J_{i^{*}}$, we have $x_{i^{*}}-x_{i^{*}}^{\prime\prime}\in\tilde{H}_{i^{*}}^{\perp}\subset[M_{1}\cap\tilde{H}_{i^{*}}]^{\perp}$,
Thus 
\[
x_{i^{*}+1}\in x_{i^{*}}+[M_{1}\cap\tilde{H}_{i^{*}}]^{\perp},
\]
or $x_{i^{*}+1}=P_{M\cap\tilde{H}_{i^{*}}}(x_{i^{*}})$. Since $x_{0}-x_{i^{*}}\in\bspan(\{a_{j}:j\in J_{i^{*}}\})$,
we have 
\[
x_{0}-x_{i^{*}}\in\tilde{H}_{i^{*}}^{\perp}\subset[M_{1}\cap\tilde{H}_{i^{*}}]^{\perp},
\]
so $x_{i^{*}+1}=P_{M_{1}\cap\tilde{H}_{i^{*}}}(x_{0})$ using a similar
argument. Since $0\in M_{1}\cap\tilde{H}_{i^{*}}$, we can deduce
\eqref{eq:x0-xi-xi-condn-1}, ending our proof by induction.
\end{proof}
It is clear that some variant of condition (A) is necessary so that
we project onto each set $M_{l}$ infinitely often, otherwise we may
converge to some point outside $M$. We now give our reasons to show
that it will be hard to prove the result if conditions (A) and (B)
were dropped. 
\begin{example}
\label{exa:p-needed}(Difficulties in dropping conditions in strong
convergence theorems) Consider the case when $k=2$. The linear operator
$P_{M_{2}}P_{M_{1}}(\cdot)$ is nonexpansive. But $M_{1}^{\perp}+M_{2}^{\perp}$
is a closed subspace if and only if $\|P_{M_{2}}P_{M_{1}}P_{M^{\perp}}\|<1$
\cite{Bauschke-Borwein-Lewis97}. We look at the case when 
\begin{equation}
\|P_{M_{2}}P_{M_{1}}P_{M^{\perp}}\|=1.\label{eq:norm-eq-1}
\end{equation}
The hyperplanes $\tilde{H}_{i}$ considered in the algorithm satisfy
$0\in\tilde{H}_{i}$. Suppose that this is the condition imposed on
the $\tilde{H}_{i}$ rather than $\tilde{H}_{i}$ being the intersection
of hyperplanes found by previous iterations. We refer to Figure \ref{fig:proj}.
The points $x_{i_{1}}$ and $x_{i_{2}}$, where $i_{1}<i_{2}$, are
iterates of Algorithm \ref{alg:accel}, and $x_{i_{2}}$ is obtained
after projecting consecutively onto four subspaces from $x_{i_{1}}$.
This arises when a third subspace $M_{3}$ is the Hilbert space $X$
and we project onto different hyperplanes passing through $0$ after
projecting onto $M_{3}$. We now show that it is possible for the
iterates $x_{i_{1}}$ and $x_{i_{2}}$ to be such that $\frac{\|x_{i_{2}}\|}{\|x_{i_{1}}\|}$
is arbitrarily close to 1. Suppose the angle $\angle x_{i_{1}}0x_{i_{2}}$
is $\theta$. If $x_{i_{2}}$ is obtained by projecting consecutively
onto $k$ subspaces, where consecutive subspaces are at an angle of
$\theta/k$. We can use Theorem \ref{thm:1-attractive} and some elementary
geometry to bound $\|x_{i_{2}}\|$ by 
\[
\|x_{i_{1}}\|^{2}\left[1-k\left[\sin\frac{\theta}{k}\right]^{2}\right]\leq\|x_{i_{2}}\|^{2}\leq\|x_{i_{1}}\|^{2}.
\]
Some simple trigonometry gives us $\lim_{k\to\infty}k[\sin\frac{\theta}{k}]^{2}=0$.
This would imply that $\frac{\|x_{i_{2}}\|}{\|x_{i_{1}}\|}$ can be
arbitrarily close to $1$ if we allow for projections onto arbitrarily
large number of subspaces containing $M$ as claimed. Combining this
fact together with \eqref{eq:norm-eq-1}, we cannot rule out that
(by our method of proof at least) it is possible that the iterates
$x_{i}$ may not even converge to $P_{M}(x_{0})$.

\begin{figure}[h]
\includegraphics[scale=0.5]{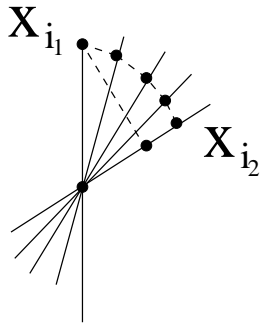}\caption{\label{fig:proj}The iterate $x_{i_{2}}$ is obtained from $x_{i_{1}}$
after projecting onto 4 linear subspaces. Note that $\|x_{i_{2}}\|$
is larger than the norm of the projection of $x_{i_{1}}$ onto the
final subspace. }

\end{figure}
\end{example}
\begin{rem}
(Connection to Dykstra's algorithm) Dykstra's algorithm \cite{Dykstra83}
is an algorithm to find the projection of a point onto the intersection
of finitely many closed convex sets (not necessarily affine subspaces).
The difference between Dykstra's algorithm and the method of alternating
projections is the additional correction vectors in Dykstra's algorithm.
Readers familiar with Dykstra's algorithm will know that in the case
of finitely many affine subspaces, Dykstra's algorithm reduces to
the method of alternating projections. The Boyle-Dykstra Theorem \cite{BD86}
proves the correctness of Dykstra's algorithm, and we have used ideas
in \cite{BD86} for our proof. A reason why we only analyze the problem
of accelerating alternating projections in the case of finitely many
affine spaces and not the more general setting of accelerating Dykstra's
algorithm is that we feel that the idea of using supporting halfspaces
and quadratic programming as explained in \cite{cut_Pang12} will
be more effective than Dykstra's algorithm in general.
\end{rem}
Finally, we remark that a consequence of our strong convergence theorems
is that strong convergence is guaranteed even when the projection
order is not cyclic. These observations have already been made in
\cite{Hundal-Deutsch-97} when they were analyzing the more general
Dykstra's algorithm.

\section{\label{sec:experiments}Performance of acceleration}

In this section, we consider a Matrix Model Updating Problem (MMUP)
as presented in \cite[Section 6.2]{EsRa11}, who in turn cited \cite{Datta_Sarkissian01,Moreno_Datta_Raydan09},
and show how one can use Algorithm \ref{alg:accel} to solve the problem.
We also show the numerical performance of our acceleration.

The problem of interest is as follows. For $M,D,K\in\mathbb{R}^{n\times n}$,
we want to solve\begin{subequations}\label{eq:MMUP} 
\begin{eqnarray}
 & \underset{\scriptsize\tilde{K},\tilde{D}\in\mathbb{R}^{n\times n}}{\min} & \|K-\tilde{K}\|_{F}^{2}+\|D-\tilde{D}\|_{F}^{2}\label{eq:MMUP1}\\
 & \mbox{s.t.} & \tilde{K}=\tilde{K}^{T},\mbox{ }\tilde{D}=\tilde{D}^{T},\label{eq:MMUP2}\\
 &  & MY_{1}(\Lambda_{1}^{*})^{2}+\tilde{D}Y_{1}(\Lambda_{1}^{*})+\tilde{K}Y_{1}=0,\label{eq:MMUP3}
\end{eqnarray}
\end{subequations}where $\Lambda_{1}^{*}\in\mathbb{C}^{p\times p}$
and $\Lambda_{1}^{*}=\mbox{diag}(\mu_{1},\dots,\mu_{p})$ and $Y_{1}\in\mathbb{C}^{n\times p}$
with columns $y_{1},\dots,y_{p}$ are the matrices of the desired
eigenvalues $\{\mu_{i}\}_{i=1}^{p}$ and eigenvectors $\{y_{i}\}_{i=1}^{p}$.
Problem \eqref{eq:MMUP} arises when we want to find minimal perturbations
in $K$ and $D$ so that some undesirable eigenvalues are moved to
more desirable values.

We can transform \eqref{eq:MMUP} as follows. We start by writing
\eqref{eq:MMUP3} as 
\[
A+\tilde{D}B+\tilde{K}C=0,
\]
where $A,B,C\in\mathbb{C}^{n\times p}$ are 
\begin{equation}
A=MY_{1}(\Lambda_{1}^{*})^{2}\mbox{, }B=Y_{1}(\Lambda_{1}^{*})\mbox{ and }C=Y_{1}.\label{eq:ABC}
\end{equation}
We can now write \eqref{eq:MMUP1} as a function of only one $2n\times2n$
block matrix variable. Define the matrices $X_{0}\in\mathbb{R}^{2n\times2n}$
and $\tilde{X}\in\mathbb{R}^{2n\times2n}$ by 
\[
X_{0}=\left(\begin{array}{cc}
K & 0\\
0 & D
\end{array}\right)\mbox{ and }\tilde{X}=\left(\begin{array}{cc}
\tilde{K} & 0\\
0 & \tilde{D}
\end{array}\right).
\]
We now write \eqref{eq:MMUP3} in terms of $\tilde{X}$. Define the
block matrices $W$ and $\hat{I}$ as 
\[
\hat{I}:=\left({I_{n\times n}\atop I_{n\times n}}\right)\mbox{ and }W=\left({C\atop B}\right),
\]
where $I_{n\times n}\in\mathbb{R}^{n\times n}$ is the identity matrix.
Note that 
\begin{eqnarray*}
A+\hat{I}^{T}\tilde{X}W & = & A+\left(\begin{array}{cc}
I_{n\times n} & I_{n\times n}\end{array}\right)\left(\begin{array}{cc}
\tilde{K} & 0\\
0 & \tilde{D}
\end{array}\right)\left(\begin{array}{c}
C\\
B
\end{array}\right)\\
 & = & A+\tilde{K}C+\tilde{D}B\\
 & = & A+\tilde{X}_{22}B+\tilde{X}_{11}C.
\end{eqnarray*}
Then problem \eqref{eq:MMUP} is reduced to that of finding the matrix
$\tilde{X}$ that solves the following optimization problem\begin{subequations}\label{eq:MMUP-v2}
\begin{eqnarray}
 & \underset{\scriptsize\tilde{X}\in\mathbb{R}^{2n\times2n}}{\min} & \|X_{0}-\tilde{X}\|_{F}^{2}\label{eq:MMUP-v2-1}\\
 & \mbox{s.t. } & \tilde{X}=\tilde{X}^{T}\mbox{, }\tilde{X}_{21}=\tilde{X}_{12}=0\label{eq:MMUP-v2-2}\\
 &  & A+\hat{I}^{T}\tilde{X}W=0.\label{eq:MMUP-v2-3}
\end{eqnarray}
\end{subequations}The projection of a matrix $X$ onto the set $S$
of matrices satisfying the first constraint \eqref{eq:MMUP-v2-2}
is given by 
\begin{eqnarray*}
 &  & [P_{S}(X)]_{12}=[P_{S}(X)]_{21}=0,\\
 &  & [P_{S}(X)]_{11}=\frac{1}{2}[X_{11}+X_{11}^{T}],\\
 & \mbox{ and } & [P_{S}(X)]_{22}=\frac{1}{2}[X_{22}+X_{22}^{T}].
\end{eqnarray*}
For the second constraint, we need to project onto the linear variety
\begin{equation}
V:=\{X\in\mathbb{R}^{2n\times2n}:A+\hat{I}XW=0\}.\label{eq:subspace-V}
\end{equation}

\begin{thm}
\cite{Moreno_Datta_Raydan09} If $X\in\mathbb{R}^{2n\times2n}$ is
any given matrix, then the projection onto the linear variety $V$
is given by 
\begin{eqnarray*}
 &  & P_{V}(X)=X+Z\Sigma W^{T},\\
 & \mbox{where } & \Sigma^{T}=-\frac{1}{2}[W^{T}W]^{-1}(A^{T}+W^{T}X^{T}Z).
\end{eqnarray*}

\end{thm}

\subsection{Numerical experiments}

We consider two algorithms for solving the MMUP problem \eqref{eq:MMUP-v2}.
In the first algorithm, we make specific choices on step 1.
\begin{algorithm}
\label{alg:MMUP-alg-1}(MMUP algorithm 1) For a starting matrix $X_{0}$,
we wish to solve \eqref{eq:MMUP-v2}. We apply Algorithm \ref{alg:accel}
by choosing the first affine space to be $S$ and the second affine
space to be $V$. We project onto $S$ and $V$ alternately, starting
with $S$. Choose $q$ to be a positive integer. The affine space
$\tilde{H}_{i}$ is chosen to be the intersection of the last $q$
affine spaces identified, or all of the affine spaces if less than
$q$ affine spaces were identified. 
\end{algorithm}
It is clear that $q=1$ corresponds to the alternating projection
algorithm.

We now describe a second algorithm for the MMUP.
\begin{algorithm}
\label{alg:MMUP-alg-2}(MMUP Algorithm 2) For a starting matrix $X_{0}\in S$,
we wish to solve \eqref{eq:MMUP-v2}. We apply Algorithm \ref{alg:accel2}
by choosing the first affine space to be $S$ and the second affine
space to be $V$. We choose $S$ to play the role of $M_{1}$ in Algorithm
\ref{alg:accel2}. Choose $q$ to be a positive integer. The affine
space $\tilde{H}_{i}$ is chosen to be the intersection of the last
$q$ affine spaces identified, or all of the affine spaces if less
than $q$ affine spaces were identified. 
\end{algorithm}
One can see that the choice of $\tilde{H}_{i}$ in Algorithms \ref{alg:MMUP-alg-1}
and \ref{alg:MMUP-alg-2} do not satisfy condition (B). Nevertheless,
if the iterates do converge, we can show that limit points must be
of the form $[x_{0}+M^{\perp}]\cap M$, and the only point satisfying
this property is $P_{M}(x_{0})$. We still obtain desirable numerical
results in our experiments.
\begin{rem}
(Sparsity in Algorithm \ref{alg:MMUP-alg-2}) Note that the iterates
$X_{i}$ and $a_{i}$, the normal vectors of the halfspaces produced,
have to lie in the space $S$, which is sparse. Besides the ease of
projection onto $S$ and the large codimension of $S$, the sparsity
of iterates and normals is another reason why Algorithm \ref{alg:MMUP-alg-2}
performs better than Algorithm \ref{alg:MMUP-alg-1}.
\begin{rem}
(The case of $q=\infty$) In our problem, the two affine spaces $S$
and $V$ are both determined by finitely many equations. We can define
both Algorithms \ref{alg:MMUP-alg-1} and \ref{alg:MMUP-alg-2} by
setting the parameter $q$ to be $\infty$. What this means is that
we project onto the affine space produced by intersecting all previous
hyperplanes generated in earlier iterations. We can converge in finitely
many iterations for both algorithms once we identify all the equations
defining the two subspaces, but the computational costs for solving
the resulting system can be huge. (The reason why the alternating
projection method is preferable is that the cost per iteration is
small.)
\end{rem}
\end{rem}
We now perform our experiments on two problems presented in \cite[Section 6.2]{EsRa11}.

\subsubsection{\label{sub:Experiment-1}Experiment 1}

For our first experiment, we choose $M,D,K\in\mathbb{R}^{4\times4}$
to be the symmetric positive definite matrices as described in \cite{Datta_Sarkissian01}:
\begin{eqnarray*}
M & = & \left(\begin{array}{cccc}
1.4685 & 0.7177 & 0.4757 & 0.4311\\
0.7177 & 2.6938 & 1.2660 & 0.9676\\
0.4757 & 1.2660 & 2.7061 & 1.3948\\
0.4311 & 0.9676 & 1.3918 & 2.1876
\end{array}\right)\mbox{, }\\
D & = & \left(\begin{array}{cccc}
1.3525 & 1.2695 & 0.7967 & 0.8160\\
1.2695 & 1.3274 & 0.9144 & 0.7325\\
0.7967 & 0.9144 & 0.9456 & 0.8310\\
0.8160 & 0.7325 & 0.8310 & 1.1536
\end{array}\right)\mbox{, }\\
K & = & \left(\begin{array}{cccc}
1.7824 & 0.0076 & -0.1359 & -0.7290\\
0.0076 & 1.0287 & -0.0101 & -0.0493\\
-0.1359 & -0.0101 & 2.8360 & -0.2564\\
-0.7290 & -0.0493 & -0.2564 & 1.9130
\end{array}\right).
\end{eqnarray*}
The eigenvalues of $P(\lambda)=\lambda^{2}M+\lambda D+K$ computed
via MATLAB are $-0.0861\pm1.6242i$, $-0.1022\pm0.8876i$, $-0.1748\pm1.1922i$
and $-0.4480\pm0.2465i$. We want to reassign only the most unstable
pair of eigenvalues, namely $-0.0861\pm1.6242i$, to the locations
$-0.1\pm1.6242i$. Let the matrix of eigenvectors to be assigned be
\[
\left(\begin{array}{cc}
1.0000 & 1.0000\\
0.0535+0.3834i & 0.0535-0.3834i\\
0.5297+0.0668i & 0.5297-0.0668i\\
0.6711+0.4175i & 0.6711-0.4175i
\end{array}\right).
\]
The formulas for $A$, $\hat{I}$ and $W$ can work in principle,
but we decide to use a different strategy when the targeted eigenvalues
and eigenvectors are complex conjugates. Consider the targeted eigenvalue
$\mu_{1}=-0.1+1.6242i$ and its targeted eigenvector 
\[
y_{1}=(\begin{array}{cccc}
1.0000 & 0.0535+0.3834i & 0.5297+0.0668i & 0.6711+0.4175i\end{array})^{T}.
\]
Instead of choosing $A$, $B$ and $C$ in the manner of \eqref{eq:ABC},
we choose $A,B,C\in\mathbb{R}^{4\times2}$ to be 
\begin{eqnarray*}
A & = & M[\begin{array}{cc}
\mbox{Re}(y_{1}\mu_{1}^{2}) & \mbox{Im}(y_{1}\mu_{1}^{2})\end{array}],\\
B & = & \phantom{M}[\begin{array}{cc}
\mbox{Re}(y_{1}\mu_{1}) & \mbox{Im}(y_{1}\mu_{1})\end{array}].\\
\mbox{ and }C & = & \phantom{M}[\begin{array}{cc}
\mbox{Re}(y_{1}\phantom{\mu_{1}^{2}}) & \mbox{Im}(y_{1}\phantom{\mu_{1}^{2}})\end{array}].
\end{eqnarray*}

We illustrate the results of this experiment in Figure \ref{fig:test}.
The experiments show that the effectiveness of Algorithm \ref{alg:MMUP-alg-1}
and Algorithm \ref{alg:MMUP-alg-2}.

\begin{figure}[h]
\includegraphics[scale=0.43]{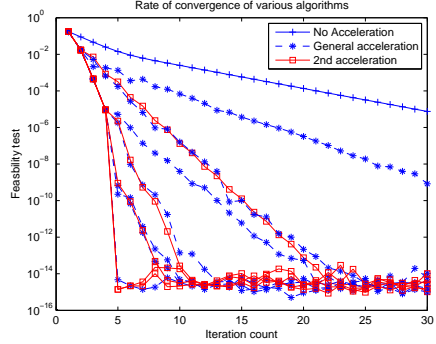}

\caption{\label{fig:test}This figure shows the results for the experiment
described in Subsubsection \ref{sub:Experiment-1}. We plot $\|A+\hat{I}^{T}\tilde{X}W\|$,
the distance of the iterates to $V$, against the number of times
we project onto $V$ \eqref{eq:subspace-V} for three different algorithms,
namely the method of alternating projections, Algorithm \ref{alg:MMUP-alg-1}
(General acceleration) and Algorithm \ref{alg:MMUP-alg-2} (2nd acceleration).
For Algorithm \ref{alg:MMUP-alg-1}, we test with parameters $q=2,3,4,5,6,7,8$.
(Note that $q=1$ corresponds to the alternating projection algorithm).
For Algorithm \ref{alg:MMUP-alg-2}, we test with parameters $q=1,2,3,4,5$.}

\end{figure}

\subsubsection{Experiment 2}

We repeat the experiment in \cite{Moreno_Datta_Raydan09} for the
case when $M,D,K\in\mathbb{R}^{30\times30}$ are the matrices {\small 
\begin{eqnarray*}
 &  & M=D=4I_{30\times30}=\left(\begin{array}{cccccc}
4 & 0 & 0 & \cdots & 0 & 0\\
0 & 4 & 0 & \cdots & 0 & 0\\
0 & 0 & 4 & \cdots & 0 & 0\\
\vdots & \vdots & \vdots & \ddots & \vdots & \vdots\\
0 & 0 & 0 & \cdots & 4 & 0\\
0 & 0 & 0 & \cdots & 0 & 4
\end{array}\right)\\
 & \mbox{ and } & K=\left(\begin{array}{cccccc}
1 & -1 & 0 & \cdots & 0 & 0\\
-1 & 2 & -1 & \cdots & 0 & 0\\
0 & -1 & 2 & \cdots & 0 & 0\\
\vdots & \vdots & \ddots & \ddots & \ddots & \vdots\\
0 & 0 & \cdots & -1 & 2 & -1\\
0 & 0 & \cdots & 0 & -1 & 1
\end{array}\right).
\end{eqnarray*}
}The pencil $P(\lambda)=\lambda^{2}M+\lambda D+K$ has 60 eigenvalues,
but the eigenvalue that causes the instability is $0$ with eigenvector
$\frac{1}{\sqrt{30}}(1,1,\dots,1)^{T}$, and the rest of the spectrum
of $P(\lambda)$ is below $-0.0027$. We use $Y_{1}=\frac{1}{\sqrt{30}}(1,1,\dots,1)^{T}$
with targeted eigenvalue $-0.018$.

Our experiments indicate that in one iteration of both Algorithms
\ref{alg:MMUP-alg-1} and \ref{alg:MMUP-alg-2}, the norm $\|A+\hat{I}^{T}\tilde{X}W\|$
goes down by a factor of $2.4\times10^{-14}$, essentially reaching
convergence within the numerical limits. For the alternating projection
algorithm, the decrease is linear, and each iteration reduces the
norm $\|A+\hat{I}^{T}\tilde{X}W\|$ by a factor of $0.5$. This experiment
once again illustrates the efficiency of the accelerations in Algorithms
\ref{alg:MMUP-alg-1} and \ref{alg:MMUP-alg-2}.

\section{Conclusion}

In this paper, we propose acceleration methods for projecting onto
the intersection of finitely many affine spaces. This strategy can
be applied to general feasibility problems where not only affine spaces
are involved, as long as there is more than one affine space.
\begin{acknowledgement*}
The author acknowledges very helpful conversations with 
Lim Chuan Li which led to Propositions \ref{prop:proj-intersection-of-2}
and \ref{prop:2-subspaces} and Algorithm \ref{alg:accel2}. He is grateful to his employer, the National University of Singapore, for his startup grant.
\end{acknowledgement*}
\bibliographystyle{amsalpha}
\bibliography{refs}

\providecommand{\bysame}{\leavevmode\hbox to3em{\hrulefill}\thinspace}
\providecommand{\MR}{\relax\ifhmode\unskip\space\fi MR }
\providecommand{\MRhref}[2]{%
  \href{http://www.ams.org/mathscinet-getitem?mr=#1}{#2}
}
\providecommand{\href}[2]{#2}
\begin{thebibliography}{BDHP03}

\bibitem[Agm83]{Agmon54}
S.~Agmon, \emph{The relaxation method for linear inequalities}, Canad. J. Math.
  \textbf{4} (1983), 479--489.

\bibitem[BBL97]{Bauschke-Borwein-Lewis97}
H.H. Bauschke, J.M. Borwein, and A.S. Lewis, \emph{The method of cyclic
  projections for closed convex sets in {H}ilbert space}, Recent developments
  in optimization theory and nonlinear analysis ({J}erusalem, 1995),
  Contemporary Mathematics 204, Amer. Math. Soc., Providence, R.I., 1997,
  pp.~1--38.

\bibitem[BCK06]{BausCombKruk06}
H.H. Bauschke, P.L. Combettes, and S.G. Kruk, \emph{Extrapolation algorithm for
  affine-convex feasibility problems}, Numer. Algorithms \textbf{41} (2006),
  239--274.

\bibitem[BD85]{BD86}
J.P. Boyle and R.L. Dykstra, \emph{A method for finding projections onto the
  intersection of convex sets in {H}ilbert spaces}, Advances in Order
  Restricted Statistical Inference, Lecture notes in Statistics, Springer, New
  York, 1985, pp.~28--47.

\bibitem[BDHP03]{BDHP03}
H.H. Bauschke, F.~Deutsch, H.S. Hundal, and S.-H. Park, \emph{Accelerating the
  convergence of the method of alternating projections}, Trans. Amer. Math.
  Soc. \textbf{355} (2003), no.~9, 3433--3461.

\bibitem[DS01]{Datta_Sarkissian01}
B.N. Datta and D.R. Sarkissian, \emph{Theory and computations of some inverse
  eigenvalue problems for the quadratic pencil}, Mathematics, Computer Science,
  and Engineering {I}, Contemporary Mathematics Volume 280, Structured
  Matrices, Amer. Math. Soc., New York, 2001, pp.~221--240.

\bibitem[Dyk83]{Dykstra83}
R.L. Dykstra, \emph{An algorithm for restricted least-squares regression}, J.
  Amer. Statist. Assoc. \textbf{78} (1983), 837--842.

\bibitem[ER11]{EsRa11}
R.~Escalante and M.~Raydan, \emph{Alternating projection methods}, SIAM, 2011.

\bibitem[GK89]{GK89}
W.B. Gearhart and M.~Koshy, \emph{Acceleration schemes for the method of
  alternating projections}, J. Comput. Appl. Math. \textbf{26} (1989),
  235--249.

\bibitem[GPR67]{GPR67}
L.G. Gubin, B.T. Polyak, and E.V. Raik, \emph{The method of projections for
  finding the common point of convex sets}, USSR Comput. Math. Math. Phys.
  \textbf{7} (1967), no.~6, 1--24.

\bibitem[Hal62]{Halperin62}
I.~Halperin, \emph{The product of projection operators}, Acta. Sci. Math.
  (Szeged) \textbf{23} (1962), 96--99.

\bibitem[HD97]{Hundal-Deutsch-97}
H.S. Hundal and F.~Deutsch, \emph{Two generalizations of {D}ykstra's cyclic
  projections algorithm}, Math. Programming \textbf{77} (1997), 335--355.

\bibitem[MDR09]{Moreno_Datta_Raydan09}
J.~Moreno, B.N. Datta, and M.~Raydan, \emph{A symmetry perserving alternating
  projection method for matrix model updating}, Mech. Syst. Signal Process
  \textbf{23} (2009), 1784--1791.

\bibitem[NT14]{NeedellTropp14}
D.~Needell and J.~A. Tropp, \emph{Paved with good intentions: Analysis of a
  randomized block {K}aczmarz method}, Linear Algebra Appl. \textbf{441}
  (2014), 199--221.

\bibitem[Pan14a]{better_alg}
C.H.J. Pang, \emph{Improved analysis of algorithms based on supporting
  halfspaces and quadratic programming for the convex intersection and
  feasibility problems}, (preprint) (2014).

\bibitem[Pan14b]{cut_Pang12}
\bysame, \emph{Set intersection problems: Supporting hyperplanes and quadratic
  programming}, Math. Programming (Online first) (2014).

\bibitem[Pie84]{Pierra84}
G.~Pierra, \emph{Decomposition through formalization in a product space}, Math.
  Programming \textbf{28} (1984), 96--115.

\bibitem[vN50]{Neumann50}
J.~von Neumann, \emph{Functional operators. {II}. {T}he geometry of orthogonal
  spaces.}, {Annals of Mathematics Studies, no. 22.}, Princeton University
  Press, Princeton, NJ, 1950, [This is a reprint of mimeograghed lecture notes
  first distributed in 1933.].

\bibitem[XZ02]{Xu_Zikatanov02}
Jinchao Xu and Ludmil Zikatanov, \emph{The method of alternating projections
  and the method of subspace corrections in {H}ilbert space}, J. Amer. Math.
  Soc. \textbf{15} (2002), no.~3, 573--597.

\end{thebibliography}

\end{document}